\documentclass[reqno, a4paper]{amsart}

\usepackage[T1]{fontenc}
\usepackage[utf8]{inputenc}
\usepackage[british]{babel}
\usepackage{enumitem}
\usepackage{amsfonts, amssymb, amsthm, amsmath}
\usepackage[all]{xy}
\usepackage{mathbbol}
\usepackage{booktabs}
\usepackage{calrsfs}
\usepackage{graphicx}
\usepackage{xcolor}
\usepackage{tikz-cd}
\usetikzlibrary{calc}
\usepackage{mathtools}
\usepackage[colorlinks, citecolor=blue]{hyperref}

\usepackage{mathabx}

\mathtoolsset{mathic}

\theoremstyle{plain}
\newtheorem{lemma}[subsection]{Lemma}

\newtheorem{theorem}[subsection]{Theorem}
\newtheorem{corollary}[subsection]{Corollary}

\theoremstyle{remark}
\newtheorem{example}[subsection]{Example}
\newtheorem{remark}[subsection]{Remark}

\theoremstyle{definition}
\newtheorem{definition}[subsection]{Definition}

\allowdisplaybreaks

\makeatletter
\newcommand*{\comment}[4][]{%
  \@ifundefined{c@#2}{\newcounter{#2}}{}%
  \begingroup
  \ifblank{#1}{%
    \newcommand*{\header}{\csname the#2\endcsname}%
  }{%
    \newcommand*{\header}{(#1\csname the#2\endcsname)}%
  }%
  \addtocounter{#2}{1}%
  \textnormal{\textcolor{#3}{\header.[#4\@]}}%
  \endgroup
}
\makeatother
\DeclareMathOperator{\eval}{eval}
\newcommand{\spn}[1]{\left\langle#1\right\rangle}
\renewcommand{\k}{\kappa}

\newcommand{\N}{\mathbb{N}}
\newtheorem*{theorem*}{Theorem}
\newcommand{\id}{\mathbb{1}}
\usepackage{relsize}
\usepackage{scalerel}
\newcommand{\cat}[1]{{\underline{\normalfont\textbf{#1}}}}
\def\acts{\curvearrowright}
\newcommand{\inje}{\hookrightarrow}

\newcommand{\surj}{\twoheadrightarrow}

\newcommand{\splitepi}{\mathrel{\substack{\mathlarger{\twoheadrightarrow}\\[-0.3ex]
			\mathlarger{\hookleftarrow}}}}

\newcommand{\splitepim}[2]{\mathrel{\substack{\overset{#1}{{\mathlarger{\twoheadrightarrow}}}\\[-0.3ex]
			{\underset{#2}{\mathlarger{\hookleftarrow}}}}}}
\newlist{tfae}{enumerate}{1}
\setlist[tfae]{label=\textnormal{(\roman*)}}

\DeclareMathOperator{\Rinv}{Rinv}
\DeclareMathOperator{\inv}{inv}

\DeclarePairedDelimiterX{\IndArr}[1]{(}{)}{
  \renewcommand*{\and}{,}
  #1
}
\DeclarePairedDelimiterX{\CoindArr}[1]{\langle}{\rangle}{
  \renewcommand*{\and}{,}
  #1
}

\tikzcdset{
  mono/.style=tail
}
\tikzcdset{
  nmono/.style=Triangle[{open, reversed}]->
}
\tikzcdset{
  repi/.style=-Triangle[open]
}
\tikzcdset{
  induced/.style=dashed
}

\def\pullback{% with thanks to Valerian Even
 \ar@{-}[]+R+<6pt,-1pt>;[]+RD+<6pt,-6pt>%
 \ar@{-}[]+D+<1pt,-6pt>;[]+RD+<6pt,-6pt>}
 
\def\pushout{%
 \ar@{-}[]+L+<-6pt,1pt>;[]+LU+<-6pt,6pt>%
 \ar@{-}[]+U+<-1pt,6pt>;[]+LU+<-6pt,6pt>}

\def\splitpullback{%
 \ar@{-}[]+R+<6pt,-.51ex>;[]+RD+<6pt,-6pt>%
 \ar@{-}[]+D+<.51ex,-6pt>;[]+RD+<6pt,-6pt>}

\def\skewpullback{%
 \ar@{-}[]+LD+<-6pt,-6pt>;[]+LDD+<-6pt,-15.5pt>%
 \ar@{-}[]+D+<-1pt,-6pt>;[]+LDD+<-6pt,-15.5pt>}

\newdir{>>}{{}*!/3.5pt/:(1,-.2)@^{>}*!/3.5pt/:(1,+.2)@_{>}*!/7pt/:(1,-.2)@^{>}*!/7pt/:(1,+.2)@_{>}}
\newdir{ >>}{{}*!/8pt/@{|}*!/3.5pt/:(1,-.2)@^{>}*!/3.5pt/:(1,+.2)@_{>}}
\newdir{ |>}{{}*!/-3.5pt/@{|}*!/-8pt/:(1,-.2)@^{>}*!/-8pt/:(1,+.2)@_{>}}
\newdir{ >}{{}*!/-8pt/@{>}}
\newdir{>}{{}*:(1,-.2)@^{>}*:(1,+.2)@_{>}}
\newdir{<}{{}*:(1,+.2)@^{<}*:(1,-.2)@_{<}}

\tikzcdset{
  mono/.style=tail
}
\tikzcdset{
  nmono/.style=Triangle[{open, reversed}]->
}
\tikzcdset{
  repi/.style=-Triangle[open]
}
\tikzcdset{
  squared/.style={sep={#1em,between origins}},
  squared/.default=4.5
}
\DeclareMathDelimiter\AMSlrcorner{\mathclose}{AMSa}{"79}{AMSa}{"79}

\begin{document}

\title[Kaluzhnin--Krasner embedding theorem for monoids]{A Kaluzhnin--Krasner Embedding theorem\\ for Schreier extensions of Monoids}
\author{Lennert De Baecke}
\email{Lennert.Jan.G.De.Baecke@vub.be}

\address{Department of Mathematics and Data Science, Vrije Universiteit Brussel, Pleinlaan 2, 1050 Elsene, Belgium.}

\date{\today}



\begin{abstract}
	We study Schreier extensions of monoids and establish a Kaluzh\-nin--Krasner embedding theorem for Schreier extensions. First, we prove that the category of monoids is not locally algebraically cartesian closed (LACC) and that a monoid is  algebraically exponentiable in the category of monoids if and only if it is a Dedekind-finite monoid. Second,  we recall that the category  of extensions of monoids is \(S\)-LACC with \(S\) the class of Schreier extensions, which defines a wreath product \(A \wr B\) for any two monoids. Finally, we prove a Kaluzhnin--Krasner embedding theorem for Schreier extensions that are not necessarily split, i.e.\ given any Schreier extension \(A \inje G \surj B\) of monoids, there is a monomorphism \(\phi_G \colon G \inje A \wr B\), which is part of a morphism of extensions. The proof adapts the classical group-theoretic argument by replacing conjugation, which requires inverses, with a substitute made available by the Schreier property, namely, the unique factorization of elements in the fibers of the projection \(p \colon G \surj B\).
\end{abstract}

\maketitle

\section*{Introduction}
The Kaluzhnin--Krasner embedding theorem, originally stated by Marc Krasner and Lev
Kaluzhnin in 1951 \cite{KmKl}, states that for any group extension of \(A\)
and \(B\), that is,\ \(A\inje G\surj B\) where \(A\) is the kernel of the second map and \(B\) is the cokernel
of the first map, there is an injective group morphism \(\phi_G\colon G\inje A\wr B\) to the
wreath product of \(A\) and \(B\). This theorem not only states the existence but also provides a construction of this map \(\phi_G\). Moreover,  as recalled in~\cite{BXTkrasner}, the wreath product induces a split extension
\[
	\begin{tikzcd}
		A^B\arrow[r,hook]&A\wr B\arrow[r,shift left,two heads]&B\arrow[l,shift left, hook],
	\end{tikzcd}
\]
which allows us to view the Kaluzhnin--Krasner theorem as an embedding of extensions. In other words, for any extension \(A\inje G\twoheadrightarrow B\), there exists a monomorphism in the category of group extensions:
\begin{equation}\label{equation:MonoExG}
	\begin{tikzcd}
		A\arrow[r,hook,"\kappa"]\arrow[d,"\phi_A"']&G\arrow[r,two heads,"p"]\arrow[d,"\phi_G"]&B\arrow[d,equal]\\
		A^B\arrow[r,hook]&A\wr B\arrow[r,shift left,two heads]&B\arrow[l,shift left, hook].
	\end{tikzcd}
\end{equation}
Here, \(\phi_A\) is defined by the universal property of the kernel \(A^B\inje A\wr B\).
This construction is not unique; each possible \cat{Set}-section \(s\colon
B\inje G\) defines such a group monomorphism \(\phi_G\colon G\inje A\wr B\). The wreath product \(A\wr B\) is the semidirect product \(A^B\rtimes B\) of
\(A^B\) (all possible \cat{Set}-maps) and \(B\), with the following action:
\[
	\acts\colon A^B\times B \to A^B\colon \left(h(-),b\right)\mapsto h(-\cdot b).
\]
The group morphism \(\phi_G\) maps an element \(g\) in \(G\) to \(\left(h_g(-),p(g)\right)\) in \(A\wr B\), where \(h_g\)
maps \(b\) to \(s(b)gs\left(bp(g)\right)^{-1}\in A\).

The study of arbitrary extensions of groups is difficult. This theorem gives a
restriction to this problem---that is, an extension must be a subgroup of \(A\wr B\). This has nice applications; for example, Weir described Sylow subgroups of Galois groups: each such subgroup arises as a group extension, and hence, by Kaluzhnin--Krasner, all such are embedded into wreath products \cite{KKimpliesSylow}. Moreover, if \(A\) and \(B\) are solvable, then so is \(A \wr B\). The
Kaluzhnin--Krasner theorem then gives an immediate proof that any group
extension of solvable groups is again solvable, since \(G\) embeds as a
subgroup of the solvable group \(A \wr B\), and subgroups of solvable groups
are solvable. Examples of categories where an analogous Kaluzhnin--Krasner theorem holds are cocommutative Hopf algebras \cite{LOTHopf} and Lie
algebras \cite{LieLAlgebrasLACC}.

Let \(\mathbb{X}\) be a pointed category, and let \(\cat{SpEx}_B(\mathbb{X})\) be the category of split extensions over \(B\in Ob(\mathbb{X})\) (see Definition \ref{def:split-exact sequence} below). Consider the forgetful functor
\[
	K_\mathbb{X}\colon \cat{SpEx}_B(\mathbb{X})\to \mathbb{X}\colon (A\inje G \splitepi B)\mapsto
	A
\]
with the obvious action on morphisms, which we won't mention explicitly, also in other instances of this functor, throughout the paper.
If for all \(B\in Ob(\mathbb{X})\) the functor \(K_{\mathbb{X}}\) has a right adjoint then we call the category \(\mathbb{X}\)
\emph{locally algebraically cartesian closed (LACC)}. This notion was introduced by Gray \cite{Gray2012,GrayPHD, BGr},   with the aim of establishing an approach to algebraic exponentiation.
For groups, the Kaluzhnin--Krasner theorem can be used to prove that the category of groups is LACC. A class of categories which are all LACC, are the categories of internal groups in any cartesian closed category
\cite{Gray2012}. Since the category of groups corresponds to the internal groups of the cartesian closed category
\cat{Set}, it also (in addition to what was mentioned earlier) follows from \cite{Gray2012} that the category of groups is LACC. Using the LACC property for groups, one can also
construct a proof of the Kaluzhnin--Krasner theorem \cite{BXTkrasner}. We will give a brief argument (Theorem \ref{theorem:LACCdedekind}) why the category of monoids
\cat{Mon} is not LACC i.e.\ for some monoids \(B\) the functor \(K_{\cat{Mon}}\colon \cat{SpEx}_B(\cat{Mon})\to \cat{Mon}\) does not have a right adjoint. However, a first result (Theorem \ref{theorem:LACCdedekind}) is a weaker statement; if we restrict ourselves to the category of Dedekind-finite monoids (see Definition \ref{definition:DKFmon} below) the functor
\begin{equation}\label{equation:LDFMon}
	K_{\cat{Mon}}\colon \cat{SpEx}_B(\cat{Mon})\to \cat{Mon}\colon \left(A\inje G \splitepi
	B\right)\mapsto A
\end{equation}
has a right adjoint if and only if \(B\) is a Dedekind-finite monoid. This proves that a monoid \(B\) is Dedekind-finite if and only if \(B\) is \emph{algebraically exponentiable} in the category of monoids (introduced in \cite[Section 8.4]{TimXaAE} and \cite{Gray2012}).
If we instead do not restrict our set of monoids, but restrict the set of split extensions \(\cat{SpEx}_B(\cat{Mon})\) to the smaller class of so-called \emph{Schreier split extensions}
\(\cat{SchSpEx}_B(\cat{Mon})\) (see Definition \ref{definition:Schreier} below), then the functor
\begin{equation}\label{equation:LMon}
	K\colon \cat{SchSpEx}_B(\cat{Mon})\to \cat{Mon}\colon \left(A\inje G \splitepi
	B\right)\mapsto A
\end{equation}
has a right adjoint for any \(B\in Ob(\cat{Mon})\).
This was first proven in \cite{AndreaMontoi}. One should observe that even though \(K\dashv R\colon \cat{Mon}\to \cat{SchSpEx}_B(\cat{Mon})\) ends up in Schreier split extensions, the functor is defined for any two monoids \(A\) and \(B\); hence, for any two monoids, we can define a wreath product of monoids, which we will call the Schreier wreath product.
Until now, it was unclear whether an analogous Kaluzhnin--Krasner embedding theorem holds for Schreier extensions (not necessarily split). The main result of this paper is the following affirmative answer to this question:
\begin{theorem*}[Kaluzhnin--Krasner theorem for Schreier extensions]
	For any Schreier extension of monoids \(0\to A\inje G \surj B \to 0\), the monoid \(G\) can be
	embedded into the Schreier wreath product via a monoid monomorphism \(\phi_G\colon G
	\inje A\wr B\). Moreover, this comes from an embedding in \(\cat{SchEx}_B(\cat{Mon})\).
\end{theorem*}
For completeness, we provide a self-contained proof of the \(S\)-LACC condition for Schreier split extensions (see the original proof in \cite{AndreaMontoi}). We will explicitly use the right adjoint \(R\) of
(\ref{equation:LMon})  to prove the Kaluzhnin--Krasner theorem for monoids. We cannot simply change the construction of the
Kaluzhnin--Krasner theorem for groups to monoids, since it uses inverses (see the construction of
\(\phi_G\) for groups recalled above); the Schreier property gives us a way to get rid of the inverses.
Moreover, our proof fails for weaker notions, such as \emph{weakly Schreier extensions}.
Schreier split extensions were introduced by Rédei \cite{RedeiSchreierMonoids}. Schreier extensions have been studied extensively, such as the Galois theory of Schreier extensions for monoids \cite{GaloisForMonoids}, and by Patchkoria in \cite{PatchCrossed}, which connects Schreier extensions to the internal categorical/crossed module machinery, and \cite{PatchCoho}, where cohomology monoids describe Schreier extensions of semimodules by monoids.
\section{Definitions and examples}
\begin{definition}\label{def:split-exact sequence}
	Let \(\mathbb{X}\) be a pointed category and \(A\), \(B\) and \(G\)  be objects in \(\mathbb{X}\). An \emph{exact sequence} is a diagram:
	\[
		\begin{tikzcd}
			A\arrow[r,hook,"\k"]&G\arrow[r,two heads, "p"]&B,
		\end{tikzcd}
	\]
	where \(\k\colon A\to G\) and \(p\colon G \to B\) are morphisms in \(\mathbb{X}\) such that \(\k=\text{ker}(p)\) and \(p=\text{coker}(\k)\)
	(in this case, we call \(G\) an \emph{extension} of \(A\) and \(B\)).

	If, in addition, we have a morphism \(s\colon B \to G\) in \(\mathbb{X}\) such that \(p\circ s=id_B\), we call the diagram a \emph{split exact sequence} (and \(G\) a \emph{split extension} of \(A\) and \(B\)):
	\[
		\begin{tikzcd}
			A\arrow[r,hook,"\k"]&G\arrow[r,two heads, "p",shift left]&B\arrow[l,shift left,"s"].
		\end{tikzcd}
	\]

	A morphism in the category of (split) extensions over \(B\) is a couple \((\phi_A,\phi_G)\) such that in the following diagram, the two squares pointing right commute:
	\begin{equation*}
		\begin{tikzcd}
			&&A \arrow[d,"\phi_A"']\arrow[rr,"\k",hook]&&G\arrow[d,"\phi_G"]\arrow[rr,two heads,"p",shift left]
			&&B\arrow[ll,hook,shift left,"s"]\arrow[d,equal]\\
			&& A' \arrow[rr,"\k'",hook]&&G'
			\arrow[rr,two heads,"p'",shift left]
			&&B\arrow[ll,hook,shift left,"s'"]
		\end{tikzcd}.
	\end{equation*}
\end{definition}
The following definition was first introduced in \cite[Definition 2.6]{MontoliAndSobra}, and the name is inspired by Schreier internal categories introduced by Patchkoria \cite{PatchCrossed}.
\begin{definition}[\text{\cite[Section 2]{MontRodTim}}, \text{\cite[Definitions 4.1 and 3.1]{MonoidsSchreier1}}]\label{definition:Schreier}
	\leavevmode \vspace{-.01cm}
	\begin{enumerate}[label=(\roman*)]
		\item A split extension of monoids is \emph{Schreier} if for every \(g\in G\) there exists a \underline{unique} element \(a_g\in A\) such that \(g=\k(a_g)\cdot
		      s\left(p(g)\right)\).
		\item An extension of monoids is \emph{Schreier} if for each \(b\in B\)  there is an
		      element
		      \(g_b\in p^{-1}(b)\) such that each \(g\in p^{-1}(b)\) can be written as a
		      \underline{unique} \(a_g\in A\) multiplied by \(g_b\), i.e.\ \(g=\k(a_g)g_b\) if \(g\in
		      p^{-1}(b)\).
		\item An extension of monoids is \emph{weakly Schreier} if for each \(b\in B\)  there is an
		      element
		      \(g_b\in G\) such that each \(g\in p^{-1}(b)\) can be written as a
		      (not necessarily unique) \(a_g\in A\) multiplied by \(g_b\), i.e.\ \(g=\k(a_g)g_b\) if
		      \(g\in
		      p^{-1}(b)\).
	\end{enumerate}
\end{definition}
In the category \cat{Grp} of groups, the split extensions are well known: the split extensions of a group \(B\) coincide with all possible semidirect products:
\[
	\begin{tikzcd}
		A\arrow[r,hook,"\k"]&A\rtimes B\arrow[r,two heads, "p",shift left]&B\arrow[l,shift left,"s"].
	\end{tikzcd}
\]
This implies that split extensions of groups are determined by  \(A\), \(B\), and an action \(B\) on \(A\) by automorphisms.
Schreier split extensions \(A\inje G \splitepi B\) behave similarly, since such an extension coincides with a monoid action of \(B\) on \(A\) by endomorphisms such that \(G\cong A\rtimes B\) \cite[3.5.\ Semidirect products]{GaloisForMonoids}.
For an arbitrary extension of groups (i.e.\ not necessarily split), there are always \cat{Set}-maps \(q\) and \(s\) such that
\[
	\begin{tikzcd}
		A\arrow[r,shift right, "j"',hook]&G\arrow[l,"q"',two heads,shift right]\arrow[r,"p",two heads,shift left]&B\arrow[l,shift left,hook,"s"],
	\end{tikzcd}
\] and \(G=A\times B\) as sets. For a Schreier extension of monoids, we have \(G=A\times B\) as sets; however, for general monoid (split-)extensions, this is not necessarily true (see Subsection \ref{subsec:Weakly} for an example of such a weakly Schreier split extension).
\begin{remark}
	\begin{enumerate}[label=(\roman*)]
		\item For a weakly Schreier extension, there is a surjective map:
		      \begin{align*}
			      A\times B  \to G\colon
			      (a,b) \mapsto \k(a)g_b.
		      \end{align*}
		\item For a Schreier extension, the above surjective map is bijective. We prove that it is injective. Suppose \(\k(a)g_b=\k(a')g_{b'}\); this would imply that
		      \(p\left(\k(a)g_b\right)=p\left(\k(a')g_{b'}\right)\), hence \(b=b'\) and therefore
		      \(g_b=g_{b'}\). So we have \(\k(a)g_b=\k(a')g_{b}\), however by definition of a
		      Schreier extension, the \(a\) part is unique; hence, \(a=a'\).
		\item For a Schreier split extension we have, as sets, that \(G\) is isomorphic with
		      \(A\times B\) via the following bijection:
		      \begin{align*}
			      G\to A \times B\colon g\mapsto (a_g,p(g)).
		      \end{align*}
		      Moreover, as said before in this case, the monoid \(G\) is isomorphic to \(A\rtimes B\) with an action by endomorphisms of \(B\) on \(A\) (see also the work of \cite[Section 3.1]{MonoidsSchreier1}, \cite[3.5.
			      Semidirect products]{GaloisForMonoids} and \cite{MontoliAndSobra}).
	\end{enumerate}
	In all three cases, we have a natural \cat{Set}-section \(s\colon B\to G\colon b\mapsto g_b\). 

	Throughout the rest of this paper, we will write \(ag\) for the action of \(a\in A\) on \(g\) in \(G\) instead of writing the more correct \(\k(a)g\).
\end{remark}
\begin{example}\label{example:Schreier}
	\begin{enumerate}[label=(\roman*)]
		\item As an example of a Schreier split extension, consider\footnote{Here, \(\mathbb{Z}_0\coloneq \mathbb{Z}\setminus\{0\}\) and \(\mathbb{N}_0\coloneq \mathbb{Z}_{>0}\).} \[\mathbb{Z}/2\mathbb{Z}\overset{\k}{\inje} (\mathbb{Z}_0,\cdot )\splitepim{p}{s} (\mathbb{N}_0,\cdot)\] with \(\kappa\colon 0,1\mapsto 1,-1\) and \(p\colon n\mapsto |n|\). We have \(\mathbb{Z}_0\cong \mathbb{Z}/2\mathbb{Z}\rtimes \N_0\).
		\item As an example of a Schreier extension which is not split, consider \[(\mathbb{N},+)\overset{\k}{\inje} (\mathbb{N},+)\overset{p}{\surj} \mathbb{Z}/2\mathbb{Z}\]
		      with \(\k\colon n\mapsto 2n \) and \(p\colon n\mapsto n\; (\text{mod }2)\). This extension is clearly not split because there is no  \cat{Mon}-morphism \(\mathbb{Z}/2\mathbb{Z}\to \mathbb{N}\).
	\end{enumerate}
\end{example}
The Schreier extensions defined in Definition \ref{definition:Schreier} are also called right Schreier extensions in \cite{ClassSchreierKernel} or right homogeneous extensions in \cite[Chapter 2]{BookSMR}. As mentioned above, they correspond to a left\footnote{The opposite on a good name: a \emph{right} action corresponds to a \emph{left} homogeneous extension. } action; when we write the category \(\cat{SchEx}_B(\cat{Mon})\), we mean right Schreier extensions.
\section{LACC property for monoids}
\subsection{The category of monoids is not LACC}
In this section, we provide a brief argument as to why the category of monoids is not LACC, that is, the functor
\begin{equation}\label{equation:KforMon}
	K_{\cat{Mon}}\colon \cat{SpEx}_B(\cat{Mon})  \to \cat{Mon}\colon
	\left(A\inje G\splitepi B  \right)                       \mapsto A
\end{equation}
does not have a right adjoint for some \(B\in Ob(\cat{Mon})\). One can verify that this functor does have a left adjoint being:
\begin{equation}\label{equation:LforMon}
	L\colon \cat{Mon}\to \cat{SpEx}_B(\cat{Mon})\colon A\mapsto \bigl(M\inje A*B\splitepi B\bigr).
\end{equation}
Note that, by definition, \(M=KL(A)\).
Here, \(*\coloneq +_\cat{Mon}\) is the coproduct (free product) in the category of monoids.

\begin{definition}\label{definition:DKFmon}
	A Dedekind-finite monoid is a monoid \(M\) that satisfies the following condition:
	\[
		\forall x,y\in M,\quad xy=\id_M\Rightarrow yx=\id_M.
	\]
	In other words, any left inverse in \(M\) is also a right inverse, and vice versa.
\end{definition}
Throughout this section, we denote \(\Rinv(B)\) as the set of right invertible elements in \(B\) and \(\inv(B)\) as the set of invertible elements.
\begin{lemma}\label{lemma:freekernelForm}
	The category of split extensions of monoids \(\cat{SpEx}_B(\cat{Mon})\) has coproducts. In particular, for split extensions
	\[
		E\coloneq (A\inje G \splitepi B)
	\]
	and
	\[
		E'\coloneq( A'\inje G'\splitepi B)
	\]
	we have
	\begin{equation}\label{equation:freededprod}
		K(E+E')\supseteq\spn{bwb^{-1}\mid w\in A*A', b\in \Rinv(B)}.
	\end{equation}
	Moreover, if \(B\) is a Dedekind-finite monoid, then (\ref{equation:freededprod}) becomes an equality.
\end{lemma}
\begin{proof}
	It is easy to see that the coproduct is constructed via the pushout
	\[
		\begin{tikzcd}
			&G\arrow[ld,two heads, shift left,"p"]\arrow[rd]&A\arrow[l,hook]&\\
			B\arrow[ru,hook,shift left]\arrow[dr,hook,shift left]&&G*_BG'\arrow[ll,dashed, "\bar{p}"]&K(E+E')\arrow[l]\\
			&G'\arrow[ul,two heads, shift left,"p'"]\arrow[ru]&A'\arrow[l,hook]&
		\end{tikzcd}
	\]
	of \(B\inje G\) and \(B\inje G'\). The inclusion (\ref{equation:freededprod}) is trivial.
	Suppose now that \(B\) is a Dedekind-finite monoid. Take \(g\) in the kernel of \(\bar{p}\) arbitrary, we prove that \(g\) is of the form \(b_1a_1b_2a_2\cdots b_na_n\) with \(a_i\) in the disjoint union \(A\dot\cup A'\) and \(b_i\in B\) such that \(b_1b_2b_3\cdots b_n=\id_B\).
	We have \(g=g_1g_2'g_3g_4'\cdots g_{n+1}g_n'\) with \(g_i \in G\) and \(g_i'\in G'\), we have
	\begin{align*}
		\id_B & =\bar{p}(g)=p(g_1)p'(g_2')\cdots p(g_{n-1})p'(g_n').
	\end{align*}
	Since \(B\) is Dedekind finite, we have that every \(p(g_i)\) and \(p'(g_i')\) is in \(\inv(B)\). Therefore, it is easy to see that \(s\left(p(g_i)^{-1}\right)g_i\in A\) and \(s'\left(p'(g_i')^{-1}\right)g_i'\in A'\). Let \(s\left(p(g_i)^{-1}\right)g_i\eqcolon a_i\) and \(s\left(p(g_i)\right)\eqcolon b_i\) hence, we have \(b_ia_i=g_i\) for \(i\) odd and \(b_ia_i=g_i'\) for even \(i\):
	\[
		g=b_1a_1b_2a_2\cdots b_na_n.
	\]
	with \(b_1b_2\cdots b_n=\id_B\).
	Since \(B\) is Dedekind-finite, it is easy to prove that every cyclic permutation of this element is also trivial, i.e.\ (for fixed \(m\in \N\)) \(b_{\overline{i+1}}b_{\overline{i+2}}\cdots b_{\overline{i+n}}=\id_B\) for any \(i\in \N\) and \(\overline{i+m}\coloneq (i+m)\; (\text{mod } n)\). Hence observe
	\begin{align*}
		g & =b_1a_1b_2a_2b_3a_3\cdots b_na_n                                                                                                                    \\
		  & =b_1a_1(b_2b_3\cdots b_nb_1)b_2a_2(b_3b_4\cdots b_nb_1b_2)b_3a_3\cdots a_{n-1}(b_nb_1\cdots b_{n-1})b_na_n                                          \\
		  & =b_1a_1b_1^{-1}(b_1b_2)a_2(b_1b_2)^{-1}(b_1b_2b_3)a_3(b_1b_2b_3)^{-1}\cdots a_{n-1}(b_1b_2\cdots b_{n-1})^{-1}a_n                                   \\
		  & =\left(b_1a_1b_1{^{-1}}\right)\left((b_1b_2)a_2(b_1b_2)^{-1}\right)\cdots \left((b_1b_2\cdots b_{n-1})a_{n-1}(b_1b_2\cdots b_{n-1})^{-1}\right)a_n.
	\end{align*}
	We see that the element \(g\) is of the form described in (\ref{equation:freededprod}).
\end{proof}
\begin{theorem}\label{theorem:LACCdedekind}
	The kernel functor \(K_{\cat{Mon}}\colon \cat{SpEx}_B(\cat{Mon})  \to \cat{Mon}\) in (\ref{equation:KforMon}) preserves coproducts if and only if  \(B\) is a Dedekind-finite monoid.
\end{theorem}

\begin{proof}
	``\(\Leftarrow\)'' Suppose \(B\) is a Dedekind-finite monoid, and suppose we have a coproduct of two extensions \(E\coloneq(A\inje G\splitepi B)\) and \(E'\coloneq( A'\inje G'\splitepi B)\):
	\[
		\begin{tikzcd}
			&G\arrow[ld,two heads, shift left,"p"]\arrow[rd]&&A\arrow[ll,hook]\arrow[dr,hook]&\\
			B\arrow[ru,hook,shift left]\arrow[dr,hook,shift left]&&G*_BG'\arrow[ll,dashed, "\bar{p}"]&K(E+E')\arrow[l]&A*A'\arrow[ll,"\psi"',dashed,bend left]\arrow[l,"\theta"',dotted]\\
			&G'\arrow[ul,two heads, shift left,"p'"]\arrow[ru]&&A'\arrow[ll,hook]\arrow[ur,hook]&
		\end{tikzcd}.
	\]
	It follows from the proof of Lemma~\ref{lemma:freekernelForm} that the coproduct is \(G*_BG'\splitepi B\), the pushout of the left square. The map \(\psi\) is constructed via the coproduct property (in \cat{Mon}), and \(\theta\) is constructed using the kernel property since \(K(E+E')\) is the kernel of \(\bar{p}\) while the image of \(\bar{p}\circ \psi\) is also zero.
	We claim that the kernel of \(\bar{p}\) is \(A*A'\),  i.e.\ \(\theta\) is an isomorphism.
	Since \(\psi\) is injective, so is \(\theta\). Suppose it is not surjective; then there is an element \(w\in G*_BG'\) for which \(\bar{p}(w)=\id_B\) while it is not contained in \(A*A'\).
	The kernel of \(\bar{p}\) is the following monoid (see Lemma \ref{lemma:freekernelForm}):
	\[K\left(E+E'\right)=\spn{bwb^{-1}\mid w\in A*A',\; b\in \Rinv(B)}\leq G*_BG.\]
	The map \(\theta\) is just the inclusion; for an arbitrary \(w=ba_1a'_1a_2a'_2\cdots a_na'_nb^{-1}\in K(E+E')\) we need to prove that it is in the image of \(\theta\). Clearly, \(ba_1b^{-1}\in K(E)=A\); hence we have
	\[
		\theta\left((ba_1b^{-1}) (ba_1'b^{-1}) \cdots (ba_nb^{-1})(ba_n'b^{-1}) \right)=ba_1b^{-1} ba_1'b^{-1} \cdots ba_nb^{-1} ba_n'b^{-1}=w.
	\]
	The last equality follows from  \(b^{-1}b=\id_B=bb^{-1}\).

	``\(\Rightarrow\)'' Suppose \(B\) is not a Dedekind-finite monoid, then there are \(p,q\in B\) such that
	\[
		pq=\id_B\neq qp.
	\]
	Now, consider the following monoids:
	\[
		X\coloneq \spn{x},\quad Y\coloneq \spn{y}.
	\]
	Since \(L\) (see (\ref{equation:LforMon})) is a left adjoint, it preserves coproducts; hence, \(L(X*Y)=L(X)+_{\cat{SpEx}_B(\cat{Mon})}L(Y)\). We will now prove that \(KL(X)*KL(Y)\ncong KL\left(X*Y\right)\).
	We have that:
	\[
		KL(X)=\spn{b_1x^{m_1}b_2x^{m_2}\cdots b_nx^{m_n}\mid n,m_i\in \mathbb{N},\;b_1b_2\cdots b_n=\id_B}\leq X*B,
	\]\[KL(Y)=\spn{b_1y^{m_1}b_2y^{m_2}\cdots b_ny^{m_n}\mid n, m_i\in \mathbb{N},\;b_1b_2\cdots b_n=\id_B}\leq Y*B,
	\]
	\[KL(X*Y)\supseteq\spn{bwb^{-1}\mid w\in X*Y,\;b\in \Rinv(B)}.
	\]
	%with \(\Rinv(B)\) being the set of right invertible elements in \(B\). This all simplifies like for \(X\) to \(KL(X)=\spn{p^nxq^n\mid n\in \mathbb{N}}\).
	Since \(KL(X)*KL(Y)\) is a coproduct, there is a unique monomorphism
	\begin{align*}
		KL(X)*KL(Y) & \inje KL(X*Y)
	\end{align*}
	which is just the inclusion. However, because \(qp\neq \id =pq\), elements such as \(pxyq\in KL(X*Y)\) are never reached. Since, it would be the image of an element \(pxb_2\cdot b_1yq\in KL(X)*KL(Y)\) such that \(b_2b_1=\id_B\) and \(pb_2=\id_B=b_1q\), it follows easily \(b_1=p\) and \(b_2=q\).
	This would not work for Dedekind-finite monoids because the left and right inverses are the same and hence \(pxqpyq=pxyq\).
\end{proof}
\begin{corollary}\label{corrollary:LACCDEDE}
	The category of monoids is not locally algebraically cartesian closed. Moreover, the functor \(K_{\cat{Mon}}\colon \cat{SpEx}_B(\cat{Mon})  \to \cat{Mon}\) from (\ref{equation:KforMon}) has a right adjoint if and only if \(B\) is a Dedekind-finite monoid.
\end{corollary}
\begin{proof}
	The first part follows by the existence of non-Dedekind-finite monoids.

	For the second part:

	``\(\Leftarrow\)'' By Theorem \ref{theorem:LACCdedekind}, \(K_{\cat{Mon}}\) preserves coproducts also the category of monoids is a variety of universal algebra. Thus, by \cite[Theorem 2.9 \& Remark 2.2]{Gray1}, the functor has a right adjoint.

	``\(\Rightarrow\)'' If \(K_{\cat{Mon}}\) has a right adjoint, then since it is a left adjoint, it preserves coproducts; hence, by Theorem \ref{theorem:LACCdedekind} \(B\) has to be a Dedekind-finite monoid.
\end{proof}
\subsection{\text{\(S\)-LACC} property for Schreier split extensions of monoids}
In this section, we recall that for an arbitrary \(B\in Ob(\cat{Mon})\) the functor \begin{align*}
	K\colon \cat{SchSpEx}_B(\cat{Mon})  \to \cat{Mon}\colon
	\left(A\inje G\splitepi B \right)                        \mapsto A
\end{align*}
has a right adjoint.
This was first shown by Martins-Ferreira, Montoli and Sobral in \cite[Proposition 4.4]{AndreaMontoi}  in 2016 (see also \cite[Definition 4.1]{AndreaMontoi} for the definition of \(S\)-LACC). For completeness, here we give a
self-contained proof.
\begin{proof}
	The existence of a right adjoint to \(K\) is equivalent as that for every
	\(X\in Ob(\cat{Mon})\) the comma
	category \(\left(K \downarrow X \right)\) has a terminal object (see, for instance, \cite[Lemma 4.5.1]{CatContext}). Consider the following Schreier split extension:
	\[
		\begin{tikzcd}
			X^B \arrow[rr,"\k_X",hook]&&X^B\rtimes B \arrow[rr,two heads,"p_X",shift left]
			&&B\arrow[ll,hook,shift left,"s_X"]
		\end{tikzcd},
	\]
	where \(X^B\) denotes the set of \cat{Set}-maps from \(B\)  to
	\(X\), with multiplication:
	\[
		h(-)\cdot_{X^B}h'(-)\coloneq (b\mapsto h(b)\cdot_Xh'(b)).
	\]
	This defines a monoid.
	We define an action of \(B\) on \(X^B\), as follows:
	\[
		\acts \colon X^B\times B \to X^B\colon  \left(h(-),b\right)\mapsto h(-\cdot b)\eqcolon h^{b}(-).
	\]
	The multiplication in \(X^B\rtimes B\eqcolon X\wr B\) occurs as follows:
	\[
		\left(h(\cdot),b \right)\cdot_{X^B\rtimes B} \left(h'(\cdot),b'\right)\coloneq
		\left(h\cdot_{X^B}h'^b,bb'\right).
	\]
	We claim that the following object of \(\left(K\downarrow X\right)\) is terminal in this category:
	\[
		\begin{tikzcd}
			X&&X^B\arrow[ll,"\eval_{\mathbb{1}_B}"] \arrow[rr,"\k_X",hook]&&X^B\rtimes B \arrow[rr,two
				heads,"p_X",shift left]
			&&B\arrow[ll,hook,shift left,"s_X"],
		\end{tikzcd}
	\]
	with \(\eval_{\id_B}\colon h(-)\mapsto h(\id_B)\).
	It is easy to verify that \(\k\), \(p\) and \(s\) are monoid morphisms, that it is a Schreier split extension, and that \(\eval_{\id_B}\) is a morphism.
	To prove that this is a terminal object, consider another object:
	\[
		\begin{tikzcd}
			&&A \arrow[lld,"f"']\arrow[rr,"\k",hook]&&G\arrow[rr,two heads,"p",shift left]
			&&B\arrow[ll,hook,shift left,"s"]\\
			X&&&&&&\\
			&&X^B\arrow[llu,"\eval_{\id_B}"] \arrow[rr,"\k_X",hook]&&X^B\rtimes B
			\arrow[rr,two heads,"p_X",shift left]
			&&B\arrow[ll,hook,shift left,"s_X"].
		\end{tikzcd}
	\]
	We have \(G=A\rtimes B\) for a certain monoid action\footnote{Such that
		the multiplication is \((a,b)\cdot_G(a',b')\coloneq (aa'^b,bb')\).} \(B\to
	\text{End}_{\cat{Mon}}(A)\) (since it is a Schreier split extension, see \cite[3.5.
		Semidirect products]{GaloisForMonoids}, \cite[Chapter 5]{BookSMR} and \cite{MontoliAndSobra}). We will now define
	\(\phi\colon A\to X^B\) which makes the following diagram commute:
	\begin{equation}\label{diagram:LACC}
		\begin{tikzcd}
			&&A \arrow[lld,"f"']\arrow[dd,"\phi"]\arrow[rr,"\k",hook]&&A\rtimes B\arrow[dd,"\phi\times
				id_B"]\arrow[rr,two heads,"p",shift left]
			&&B\arrow[ll,hook,shift left,"s"]\arrow[dd,equal]\\
			X&&&&&&\\
			&&X^B\arrow[llu,"\eval_{\id_B}"] \arrow[rr,"\k_X",hook]&&X^B\rtimes B
			\arrow[rr,two heads,"p_X",shift left]
			&&B\arrow[ll,hook,shift left,"s_X"]
		\end{tikzcd}.
	\end{equation}
	We will see that \(\phi\) is uniquely determined (hence, our object is terminal).  Denote \(h_a\coloneq \phi(a)\). To
	make sure the triangle in diagram (\ref{diagram:LACC}) commutes, we have to define \(\phi(a)\) such that \(f(a)= \eval_{\id_B}\left(\phi(a)\right)\),
	hence \(h_a\colon \id_B\mapsto f(a)\). The first rectangle also has to
	commute, which implies \(\phi\) must be equivariant for the action of \(B\). Hence,
	\[
		h_a^{\tilde{b}}=h_{a^{\tilde{b}}} \quad \text{for any \(\tilde{b}\in B\)}.
	\]
	Knowing this, \(h_a\) is completely determined since
	\begin{align*}
		\hspace{-0.6cm}h_a(b) & =h_a(
		\id_B\cdot b)=\eval_{\id_B}(h_a(- b))=\eval_{\id_B}\left(h_a^b(-)\right)                  \\
		                      & =\eval_{\id_B}\left(h_{a^b}(-)\right)=h_{a^b}(\id_B)=f(a^b)\in X.
	\end{align*}
	One easily verifies that \(\phi\) and \(\phi\times id_B\) are \cat{Mon}-morphisms.
	Hence, we have a right adjoint of \(K\) being:
	\[
		R\colon \cat{Mon}  \to \cat{SchSpEx}_B(\cat{Mon})\colon
		X    \mapsto \left(X^B\inje X^B\rtimes B \splitepi B\right). \qedhere
	\]
\end{proof}
The previous theorem also holds for general split extensions if \(B\) is a group (in addition to Corollary \ref{corrollary:LACCDEDE}, since every group is a Dedekind-finite monoid), because if \(B\) is a group, then \(\cat{SchSpEx}_B(\cat{Mon})=\cat{SpEx}_B(\cat{Mon})\) (which can be easily verified).
\begin{definition}
	For any two monoids \(A\) and \(B\), we call the monoid defined in the previous proof \(A^B\rtimes B\eqcolon A\wr B\), the \emph{Schreier wreath product} of \(A\) and \(B\).
\end{definition}
An equally interesting object is the reverse wreath product, see Remark \ref{remark:naKK}.
\section{Kaluzhnin--Krasner for Schreier extensions}
In this section, we prove the Kaluzhnin--Krasner embedding theorem for Schreier extensions (which are not necessarily split), that is,\
\begin{theorem}\label{theorem:KKforSchreier}
	For any Schreier extension of monoids \(0\to A\to G \to B \to 0\), the monoid \(G\) can be
	embedded into the Schreier wreath product via a monoid monomorphism \(\phi_G\colon G
	\to A\wr B\).
\end{theorem}
\begin{proof}[Sketch of proof:]\renewcommand{\qedsymbol}{$\triangle$}
	The same map as for groups does not directly work because it requires inverses. However, conjugation can be performed without using inverses. The map \(a\mapsto gag^{-1}\eqcolon \bar{a}\) corresponds to the identity map \(gA\to Ag\colon ga\mapsto \bar{a}g\). For submonoids, cosets do not need to be in bijection with one another, but for Schreier extensions, we have a suitable substitute, which we exploit below.
\end{proof}
\begin{proof}
	Consider an arbitrary Schreier extension:
	\[
		A\inje G \surj B.
	\]
	We prove that \(A^B\inje A^B\rtimes B\splitepi B\) is a universal receptor, i.e.\ we
	need to find monoid morphisms \(\phi_A\) and \(\phi_G\) such that the following
	commutes, where \(s\) is just a ``natural'' \cat{Set}-section which we will define
	later:
	\begin{equation}\label{diagram:KKMon}
		\begin{tikzcd}
			A \arrow[rr,"\k",hook]\arrow[dd,"\phi_A"]&&G\arrow[dd,"\phi_G"]\arrow[rr,two
				heads,"p",shift left]
			&&B\arrow[ll,hook,shift left,"s",dashed]\arrow[dd,equal]\\
			&&&&\\
			A^B \arrow[rr,"\k_A",hook]&&A^B\rtimes B
			\arrow[rr,two heads,"p_A",shift left]
			&&B\arrow[ll,hook,shift left,"s_A"]
		\end{tikzcd}.
	\end{equation}
	We are inspired by the construction of the classical Kaluzhnin--Krasner theorem.
	The next step is to find an appropriate alternative for \(b a  b^{-1}\), which no longer makes sense. However, any Schreier extension satisfies the next property: for each \(b\in B\), there is an element \(g_b\in p^{-1}(b)\) such that for any other element
	\(g\in p^{-1}(b)\) there is a unique element \(a_g\in A\) for which \(a_gg_b=g\). In
	case \(b=\id_B\) we can choose \(g_{\id_B}\coloneq \id_G\). Hence, there is a section (which is just a \cat{Set} map):
	\begin{equation}\label{equation:defS}
		s\colon B\inje G\colon b\mapsto g_b.
	\end{equation}
	This property also defines an inclusion map (which is in fact a bijection):
	\begin{align*}
		p^{-1}(b)\to Ag_b\colon g\mapsto a_gg_b=a_gs(b),\quad \text{where \(g=a_gg_b\).}
	\end{align*}
	Because \(s(b)A\subseteq p^{-1}(b)\), there is a map:
	\begin{align*}
		s(b)A\to As(b)\colon s(b)a\mapsto \bar{a}s(b),
	\end{align*}
	This map is just the inclusion map, but now we focus on \(\bar{a}\) which is unique with the property that \(s(b)a=\bar{a}s(b)\). Hence the element \(s(b)\) provides us with a well-defined map \(A\to A\colon a\mapsto \bar{a}\).
	This defines an ``action''\footnote{This is not really an action since there is no reason to assume that \((a^b)^{b'}=a^{bb'}\).} of \(B\) on \(A\). Since
	\(\bar{a}\) is unique by the Schreier property, this is well-defined; hence
	\(\bar{a}\eqcolon a^b\). Now, we define \(\phi_A\) as:
	\begin{align*}
		\phi_A \colon A \to A^B\colon a\mapsto \left(h_a\colon b\mapsto a^b\right).
	\end{align*}
	We must verify that \(\phi_A\) is a monoid morphism. We have
	\begin{align*}
		\phi_A(aa')(b)= (aa')^b
		\iff
		s(b)aa'=(aa')^bs(b)
	\end{align*}
	for all \(b\in B\). However, for such a fixed \(b\), we have \(s(b)aa'=a^bs(b)a'=a^ba'^bs(b)\).
	Therefore, \((aa')^b=a^ba'^b\) because, by the Schreier property, there is only one element in \(A\) such that this element multiplied by \(s(b)\) gives \((aa')^bs(b)\). Hence
	\(\phi_A(aa')(b)=(aa')^b=a^ba'^b=\left(\phi_A(a)\phi_A(a')\right)(b)\). Now we define
	\(\phi_G\):
	\begin{align*}
		\phi_G \colon G \to A^B\rtimes B\colon g\mapsto \left( h_{g},p(g)\right),
	\end{align*}
	where we define \(h_g\) as follows:
	\begin{align*}
		h_g(b)\coloneq  a_{g,b},
	\end{align*}
	where \(a_{g,b}\) is the following well-defined element: consider \(g_bg\in G\) (here
	\(g_b\coloneq s(b)\)).  The image under \(p\) of this element is \(p(g_bg)=bp(g)\).
	Hence, by definition, \(s(bp(g))\coloneq g_{bp(g)}\). By
	the Schreier property there is a unique element \(a_{g,b}\) for which
	\(a_{g,b}\cdot g_{bp(g)}=g_b\cdot g\in p^{-1}(bp(g))\).
	We prove that \(\phi_G\) is a morphism:
	\begin{align*}
		\phi_G(gg')(b)=(h_{gg'},p(gg'))=(h_gh_{g'}^{p(g)},p(g)p(g'))\coloneq
		(h_g,p(g))(h_{g'},p(g')),
	\end{align*}
	where the second equality holds because \(p\) is a morphism. We need to prove for arbitrary \(b\in B\) that
	\(h_{gg'}(b)\overset{?}{=}h_g(b)\cdot h_{g'}^{p(g)}(b)=h_g(b)\cdot
	h_{g'}(bp(g))\). We have
	\begin{align*}
		h_{gg'}(b)=a_{gg',b}                                & \Leftrightarrow a_{gg',b}g_{bp(gg')}=g_bgg', \\
		h_g(b)\cdot h_{g'}(bp(g))=a_{g,b}\cdot a_{g',bp(g)} & \Leftrightarrow
		a_{g,b} g_{bp(g)}=g_bg
	\end{align*}
	and \(a_{g',bp(g)}g_{bp(g)p(g')}=g_{bp(g)}g'\). Multiplying both sides of the second equality by \(a_{g,b}\), we find:
	\begin{align*}
		a_{g,b}\cdot a_{g',bp(g)}
		g_{bp(g)p(g')} & =a_{g,b}g_{bp(g)}g'. \\
		\intertext{Then using that \(a_{g,b} g_{bp(g)}=g_bg
			\), we see:} a_{g,b}\cdot a_{g',bp(g)}
		g_{bp(gg')}    & =g_bgg'.
	\end{align*}
	By the definition of Schreier extension \(a_{gg',b}\) is the unique such that \(a_{gg',b}g_{bp(gg')}=g_bgg'\), hence
	\(a_{gg',b}=a_{g,b}\cdot a_{g',bp(g)}\). So we have proved that \(\phi_G\) is a morphism.
	We still have to verify that the two squares pointing to the right in the diagram (\ref{diagram:KKMon}) commute.
	The second square trivially commutes by the definition of \(\phi_G\). Now we have to prove
	\(\phi_G\circ \k=\k_A\circ\phi_A\), i.e.\ that for all \(a\in A\),
	\begin{align*}
		\k_A(\phi_A(a))=(h_a,\id_B) & =(h_{\k(a)},p(\k(a)))=\phi_G(\k(a))
	\end{align*}
	or, equivalently,
	\begin{align*}
		(h_a,\id_B) & =(h_{\k(a)},\id_B)
	\end{align*}
	which holds if
	\begin{align*}
		\forall b\in B \colon h_a(b) & =h_{\k(a)}(b).
	\end{align*}
	We have that \(h_a(b)\) is the unique element \(a^b\) for which \(s(b)a=a^bs(b)\) while
	\(h_{\k(a)}(b)\) is the unique element \(a_{\k(a),b}\) for which	\(s(b)\k(a)=a_{\k(a),b}\cdot g_{bp(\k(a))}\). However, since \(a_{\k(a),b}\cdot
	g_{bp(\k(a))}=a_{\k(a),b}g_{b\id_B}=a_{\k(a),b}g_b=a_{\k(a),b}s(b)\), we have
	\(h_a(b)=a^b=a_{\k(a),b}=h_{\k(a)}(b)\). Finally, we prove that \(\phi_G\) is a monomorphism, or
	equivalently, injective as a map. Suppose
	\((h_g,p(g))=(h_{g'},p(g'))\), this implies \(p(g)=p(g')\) and \(h_g=h_{g'}\). So for every \(b\) we have \(a_{g,b}=a_{g',b}\) and hence
	for all \(b\in B\) we have \(g_bg=a_{g,b}g_{bp(g)}=a_{g',b}g_{bp(g')}=g_bg'\) for
	every \(b\). Now choose \(b=\id_B\); we then\footnote{Remember we were able to
		choose \(s(\id_B)=g_{\id_B}\coloneq \id_G\).} have:
	\[
		g=\id_Gg=s(\id_B)g=g_{\id_B}g=g_{\id_B}g'=s(\id_B)g'=\id_Gg'=g'.\qedhere
	\]

\end{proof}
\begin{remark}\label{remark:naKK}
	\begin{enumerate}[label= (\roman*)]
		\item Notice that we only have proven Theorem \ref{theorem:KKforSchreier} for right Schreier extensions; however, the same holds, mutatis mutandis, for left Schreier extensions (or left homogeneous extensions), however, with a different recipient than \(A^B\rtimes B\), namely \(A^B\overline{\rtimes} B\) the reverse\footnote{Corresponding to the right action: \(A\overline{\rtimes} B\) with multiplication \((a,b)\cdot (a',b')\coloneq (a^{b'}a',bb')\) and \((a^{b})^{b'}\coloneq a^{(bb')}\), while for the left action \((a^{b})^{b'}= a^{(b'b)}\).} semidirect product (name first used in \cite{LmARSemi}). Similarly, one could define the reverse Schreier wreath product \(A\overline{\wr} B\coloneq A^B\overline{\rtimes}B\) using the right adjoint of the kernel functor of left Schreier extensions
		      \[
			      K\colon \cat{L-ScSpEx}_B(\cat{Mon})\to \cat{Mon}.
		      \]
		\item The proof of Theorem \ref{theorem:KKforSchreier} fails for weakly Schreier extensions; we could still define an ``action'' of \(B\) on \(A\); however \(\phi_A\) is not necessarily a morphism since it is not clear whether \((aa')^b=a^ba'^b\)  for all \(a\), \(a'\in A\) and \(b\in B\).
		\item In the proof of Theorem \ref{theorem:KKforSchreier}, the morphism \(\phi_G\colon G\inje A\wr B\) is defined using the choice of the section \(s\colon b\mapsto g_b\) (see (\ref{equation:defS})), however this element \(g_b\) does not need to be unique to satisfy the definition of a Schreier extension (Definition \ref{definition:Schreier}). Hence, the morphism \(\phi_G\) does not need to be unique (just as for groups). However, unlike for groups, for monoids the proof fails for \(s\) an arbitrary section.
		\item Just as for groups, the Kaluzhnin--Krasner theorem allows us to get a monomorphism of group extensions (see (\ref{equation:MonoExG}) and \cite{BXTkrasner}). Our proof of the Kaluzhnin-Krasner theorem for Schreier extensions gives us this result for monoids.
	\end{enumerate}
\end{remark}
\begin{example}
	We continue from Example \ref{example:Schreier} (ii), we construct a monomorphism \(\phi_G\colon (\mathbb{N},+)\to \mathbb{N}^{\mathbb{Z}/2\mathbb{Z}}\rtimes \mathbb{Z}/2\mathbb{Z}\). The section \(s\) defined in the proof of Theorem \ref{theorem:KKforSchreier}  gives us \(s\colon \mathbb{Z}/2\mathbb{Z}\to \mathbb{N}\colon 0,1\mapsto 0,1\) and as an embedding into \(\mathbb{N}^{\mathbb{Z}/2\mathbb{Z}}\rtimes \mathbb{Z}/2\mathbb{Z}\) the following:
	\begin{align*}
		\phi_G\colon n\mapsto\left( \begin{cases}
				                            \left(h_n\colon 0,1\mapsto \frac{n}{2},\frac{n}{2}\right)      & n\equiv 0\; \text{(mod 2)} \\
				                            \left(h_n\colon 0,1\mapsto  \frac{n-1}{2},\frac{n+1}{2}\right) & n\equiv 1\; \text{(mod 2)}
			                            \end{cases},n\; \text{(mod 2)}\right).
	\end{align*}
	Since \(h_n(\varepsilon)\) is the unique element \(a_{n,\varepsilon}\) in the kernel \(\N\) such that \(\k(a_{n,\varepsilon})+s(\varepsilon+ p(n))=s(\varepsilon)+n\).
\end{example}
\subsection{\text{Theorem \ref{theorem:KKforSchreier}} cannot be generalized to weakly Schreier extensions}\label{subsec:Weakly}
In this section, we write \((\overline{\mathbb{N}},+)\) for a copy of \((\mathbb{N},+)\), and consider the following extension:
\begin{equation}\label{equation:weakly}
	\begin{tikzcd}
		(\N,+)\arrow[r,"\kappa",hook]&\left(\N\sqcup_{0} \overline{\N},*\right)\arrow[r,"\pi",two heads]&(\overline{\N},+)
	\end{tikzcd},
\end{equation}
where \(\kappa\) is the inclusion and \(\pi\) is the projection. The monoid\footnote{In this monoid we identify the two units of the disjoint monoids, i.e.\ \(0=\overline{0}\).} \(\left(\N\sqcup_{0} \overline{\N},*\right)\) has the following operation:
\[
	n*m=n+m,\quad \overline{n}*\overline{m}=\overline{n+m},\quad  \text{if \(\overline{n}\neq 0\): }\;\overline{n}*m=\overline{n}=m*\overline{n}.
\]
By \cite{FaulWeak} such a construction yields a weakly Schreier extension.
We prove that there is no monomorphism from \(\N\sqcup_{0} \overline{\N}\) to \(\N^{\overline{\N}}\rtimes \overline{\N}\). If there were an embedding, there would be two nontrivial elements \((h,\overline{n})\) and \((g,\overline{m})\) in \(\N^{\overline{\N}}\rtimes \overline{\N}\), such that if we add them, we again have \((g,\overline{m})\) (since there are elements that do this in \(\N\sqcup_0\overline{\N}\) e.g.\ \(2*\overline{5}=\overline{5}\)), hence
\begin{align*}
	(g,\overline{m})=(h,\overline{n})\cdot (g,\overline{m}) & =\left(h(-)+g(-+\overline{n}),\overline{n}+\overline{m}\right)
\end{align*}
which implies \(\overline{n}=0\) and, finally, \(h(-)=\id_{\N^{\overline{\N}}}\). The last equality comes from the fact that \(\N\) and hence \(\N^{\overline{\N}}\) allows for cancellation. We obtain a contradiction because we assumed that \((h,\overline{n})\) was non-trivial.
However, this only proves that we cannot expand Theorem \ref{theorem:KKforSchreier} to weakly Schreier extensions; it may be possible to find a different wreath product that can be used for a Kaluzhnin--Krasner theorem. Moreover, in our example \((\overline{\N},+)\) is a Dedekind-finite monoid; hence, by Theorem \ref{theorem:LACCdedekind} we proved that the functor (\ref{equation:LMon}) has a right adjoint \(R\) with \(B=\overline{\N}\). Since (\ref{equation:weakly}) is a split extension, there is a unique morphism in \(\cat{SpEx}_B(\cat{Mon})\) from \(\N\inje \N\sqcup \overline{\N} \splitepi \overline{\N}\) to \(R(\N)\), while \(R(\N)\) defines a different wreath product. However, we did not give an explicit construction, this is more involved (see \cite[Theorem 2.9]{Gray1}).
\section{Further questions and remarks}\label{Section Further Questions}

An obvious question is whether we can generalize this to other categories, such as semigroups, near-rings, semi-rings and rings.\ However, the category of semigroups is not pointed (hence kernels and cokernels are not well defined).
For semi-rings, a similar equivalence between actions and Schreier extensions holds \cite{MontoliAndSobra}; however, the authors of that article prove that the category of semi-rings is not LACC \cite{AndreaMontoi}, not even if we restrict to Schreier extensions \cite[Proposition 6.1.6]{BookSMR}.
The category of (non-unital) rings is not LACC \cite[Proposition 6.7]{Gray2012}, but it could still be possible that there is a Kaluzhnin--Krasner theorem for (non-unital) rings. This would be surprising since, at least for groups, LACC is ``equivalent'' to the Kaluzhnin--Krasner theorem, in the sense that either can be deduced from the other. For a semi-abelian category the LACC property implies a Kaluzhnin--Krasner embedding theorem \cite[Theorem 3.3]{BXTkrasner}.
The category of V-groups is LACC \cite{Vgroups}; an extension of the Kaluzhnin--Krasner theorem to this context will be the subject of a subsequent paper.
\section*{Acknowledgements}
Thanks to Andrea Montoli, who posed this question, and to both Mark Sioen and Tim Van der Linden for guidance in writing this paper.

% \bibliographystyle{amsalpha-nodash-init-nosentcase}
% \bibliography{biblio}

\end{document}